\documentclass[a4paper,10pt]{amsart}

\usepackage{amsmath}
\usepackage{amssymb}
\usepackage{amsthm}
\usepackage{upgreek} 
\usepackage[all,cmtip]{xy} 

\newtheorem{theorem}{Theorem}[section]
\newtheorem{lemma}[theorem]{Lemma}

\newcommand{\Z}{\mathbb Z}
\newcommand{\ZZ}{{\mathbb Z\times\mathbb Z}}
\newcommand{\pset}{\{0,1,\ldots,p-1\}}
\newcommand{\gset}{\{0,1,\ldots,\gcd(m,n)-1\}}
\newcommand{\mset}{\{0,1,\ldots,m-1\}}
\newcommand{\nset}{\{0,1,\ldots,n-1\}}
\DeclareMathOperator{\lcm}{lcm}

\begin{document}

\title{A note on the no-three-in-line problem on a torus}

\author[A. Misiak]{Aleksander Misiak}
\address{
School of Mathematics, West Pomeranian University of Technology\\
al. Piast\'{o}w 48/49, 70-310 Szczecin, Poland}
\email{Aleksander.Misiak@zut.edu.pl}

\author[Z. St\c{e}pie\'{n}]{Zofia St\c{e}pie\'{n}}
\address{
School of Mathematics, West Pomeranian University of Technology\\
al. Piast\'{o}w 48/49, 70-310 Szczecin, Poland}
\email{stepien@zut.edu.pl}

\author[A. Szymaszkiewicz]{Alicja Szymaszkiewicz}
\address{
School of Mathematics, West Pomeranian University of Technology\\
al. Piast\'{o}w 48/49, 70-310 Szczecin, Poland}
\email{alicjasz@zut.edu.pl}

\author[L. Szymaszkiewicz]{Lucjan Szymaszkiewicz}
\address{
Institute of Mathematics, Szczecin University\\ 
Wielkopolska 15, 70-451 Szczecin, Poland
}
\email{lucjansz@wmf.univ.szczecin.pl}

\author[M. Zwierzchowski]{Maciej Zwierzchowski}
\address{
School of Mathematics, West Pomeranian University of Technology\\
al. Piast\'{o}w 48/49, 70-310 Szczecin, Poland}
\email{mzwierz@zut.edu.pl}

\subjclass[2000]{05B99}

\keywords{Discrete torus; No-three-in-line problem; Chinese Remainder Theorem}

\thanks{Corresponding author: Z. St\c{e}pie\'{n}; e-mail: stepien@zut.edu.pl; Tel/fax:+48914494826}

\begin{abstract}
In this paper we 
show that at most $2 \gcd(m,n)$ points can be placed with no
three in a line on an $m\times n$ discrete torus.
In the situation when $\gcd(m,n)$ is a prime, we completely solve the problem.
\end{abstract}

\maketitle

\section{Introduction}

 The no-three-in-line-problem \cite{Dud} asks for the maximum number of points that can be placed in the 
$n \times n$ grid with no three points collinear. This question has been widely studied,
 but is still not resolved.

The obvious upper bound is $2n$ since one can put at most two points in each row.
This bound is attained for many small cases, for details see \cite{F92} and \cite{F98}.
In \cite{GK} the authors  give a probabilistic argument to support the conjecture
that for a large $n$ this limit is unattainable.

As a lower bound, Erd\"os' construction (see \cite{E}) shows that for $p$ prime one can select $p$ points
with no three collinear.
In \cite{HJSW} it is shown, that for $p$ prime one can select $3(p-1)$ points from a $2p\times 2p$ grid
with no three collinear.

In the literature we can find some extensions of the no-three-in-line problem (see \cite{Fowler}, \cite{PA}).
This paper is generalization of \cite{Fowler}, where authors analyze the no-three-in-line-problem on the
discrete torus. This modified problem is still interesting.

Let $m$ and $n$ be positive integers greater than $1$. By a discrete torus $T_{m\times n}$ we mean $\mset\times\nset$.

Four integers $a,$ $b,$ $u,$ $v$ with $\gcd(u,v)=1$
correspond to the line $\left\{(a+u k, b+v k):k\in\Z  \right\}$ on $\ZZ$. The condition $\gcd(u,v)=1$ ensures that
each pair $P,Q$ of distinct points in $\ZZ$
belongs to exactly one line. For instance, the points $O=(0,0),$ $P=(2,2)$ 
belong to the line $\left\{(k,k):k\in\Z  \right\}$. 

We define lines on $T_{m\times n}$ to be images of lines in the $\mathbb Z\times \mathbb Z$
under the projection $\pi_{m,n}:\ZZ\to T_{m\times n}$ defined as follows
\[
\pi_{m,n}(a,b):=(a \bmod m, b \bmod n).
\]

By $x \bmod y$ we mean the smallest non-negative remainder when $x$ is divided by $y$.

We say that a set $X\subset T_{m\times n}$ satisfies the no-three-in-line condition if
there are no three collinear points in $X$.
Let $\uptau\left(T_{m\times n}\right)$ denote the size of 
the largest set $X$ satisfying the no-three-in-line condition.

In our paper we will prove the following theorems.
\begin{theorem} \label{M1}
We have 
\[
\uptau\left(T_{m\times n}\right)\leq 2\gcd(m,n).
\]
\end{theorem}

\begin{theorem} \label{M2}
We have
\begin{enumerate}
\item For $\gcd(m,n)=1,$ $\uptau\left(T_{m\times n}\right)= 2$. 
\item For $\gcd(m,n)=2,$ $\uptau\left(T_{m\times n}\right)= 4$.
\item  Let $\gcd(m,n)=p$ be an odd prime. 
\begin{enumerate}
\item If $\gcd(pm,n)=p^2$ or $\gcd(m,pn)=p^2,$ then $\uptau\left(T_{m\times n}\right)=2p.$ 
\item If $\gcd(pm,n)=p$ and $\gcd(m,pn)=p$, then $\uptau\left(T_{m\times n}\right)= p+1$.
\end{enumerate}
\end{enumerate}
\end{theorem}

The Theorem 1.2(1) was proved in \cite{Fowler}  by some algebraic argument.
The Theorem 1.2(2) is a generalized version of Proposition 2.1 from \cite{Fowler}. 
Similarly, Theorem 1.2(3a) and Theorem 1.2(3b) are generalizations of Theorem 2.7 and Theorem 2.9 from \cite{Fowler}, 
respectively. 

\section{Proofs of Theorem 1.1 and Theorem 1.2(1)}

One of the main tools used in this
paper is the Chinese Remainder Theorem.
\begin{theorem}[Chinese Remainder Theorem]\label{T:CRT}
Two simultaneous congruences
\[
\begin{aligned}
x&\equiv a \pmod m, \\
x&\equiv b \pmod n 
\end{aligned}
\]
are solvable if and only if $a\equiv b \pmod{\gcd(m,n)}$.
Moreover, the solution is unique modulo $\lcm(m, n)$.
\end{theorem}

Let us define the family $\mathcal L =\left\{ L_{s}:s\in\gset  \right\}$ of lines on $T_{m\times n},$ where
\[
L_{s}=\left\{\pi_{m,n}(k, k-s) \in T_{m\times n} :k\in\Z  \right\}.
\]
\begin{lemma}[see \cite{SSZ}] \label{L:Diag} 
Let 
$a=(a_{x},a_{y})\in T_{m\times n}$ and $d=(a_x - a_y) \bmod \gcd(m,n)$. Then 
$a\in L_d$.
Moreover, we have $ L_{s_{1}}\cap L_{s_{2}}=\emptyset$
for $s_1\neq s_2$ and $s_{1},s_{2}\in \gset.$

\end{lemma}
\begin{proof}
By Theorem \ref{T:CRT} there exists  $k\in \Z$ such that
\[
\begin{aligned}
k&\equiv a_{x}    &\pmod m, \\
k&\equiv a_{y} 
+d
&\pmod n.
\end{aligned}
\]
Consequently, 
$(a_{x},a_{y}) = \pi_{m,n}(k, k-d) \in L_d.$
Suppose that $L_{s_{1}}\cap L_{s_{2}}\neq\emptyset.$ This means that there are $k_{1}$, $k_{2}\in \Z$ 
such that 
$\pi_{m,n}(k_{1},k_{1}-s_{1}) = \pi_{m,n}(k_{2},k_{2}-s_{2}).$ 
In other words $k_{1}-k_{2}$ 
is the solution of the following system
\[
\begin{aligned}
k_{1}-k_{2} &\equiv 0    &\pmod m, \\
k_{1}-k_{2}&\equiv s_{1}-s_{2}   &\pmod n.
\end{aligned}
\]
By Theorem~\ref{T:CRT} again, we see that 
$s_{1}-s_{2} \equiv 0 \pmod{\gcd(m,n)}$.
Hence $L_{s_{1}}\cap L_{s_{2}}=\emptyset$ for $s_{1}\neq s_{2}$ and
$s_{1},s_{2}\in \gset.$
\end{proof}

\begin{proof}[Proof of Theorem \ref{M1}]
Let $X\subset T_{m\times n}$ satisfy the no-three-in-line condition. 
By Lemma \ref{L:Diag} for every $a\in X$ there exists $L\in\mathcal L$ such that $a\in L$.
Consequently, $\uptau(T_{m\times n})\leq 2\cdot |\mathcal L| = 2 \cdot \gcd(m,n)$.
\end{proof}

\begin{proof}[Proof of Theorem \ref{M2}(1)]
Obviously it is always true that $\uptau(T_{m\times n})\geq 2$. 
By Theorem \ref{M1} we get the statement.
\end{proof}

\section{Proofs of Theorem \ref{M2}(2) and Theorem \ref{M2}(3a)}

Let $X=(x_{1},x_{2})\in\ZZ$, $Y=(y_{1},y_{2})\in\ZZ$, $Z=(z_{1},z_{2})\in\ZZ$. 
Denote by $D(X,Y,Z)$ 
the following determinant 
\[
\begin{vmatrix}
1 & 1 & 1 \\
x_{1} & y_{1} & z_{1} \\
x_{2} & y_{2} & z_{2} \\
\end{vmatrix}
\] 

Recall the determinant criterion for checking whether points are in a line:
\begin{lemma} \label{L:D0}
Three points $X,Y,Z\in\ZZ$ are in a line if and only if $D(X,Y,Z)=0$.
\end{lemma}

Now we prove the determinant criterion on a torus.

\begin{lemma} \label{L:D}
If three points $a,$ $b$ and $c$ of $T_{m\times n}$ are in a line, then
$D(a,b,c)\equiv 0 \pmod{\gcd(m,n)}.$
\end{lemma}

\begin{proof}
Suppose that three points $a=(a_{x},a_{y}),$ $b=(b_{x},b_{y})$ and $c=(c_{x},c_{y})$ are in a line on $T_{m\times n}$.
This means that there are $A, B, C\in\ZZ$ such that $\pi(A)=a,$ $\pi(B)=b,$ $\pi(C)=c$ and $D(A,B,C)=0.$ More precisely 
\[
\begin{aligned}
A&=(a_{x}+\alpha_{x} m,a_{y}+\alpha_{y}n), \\
B&=(b_{x}+\beta_{x} m,b_{y}+\beta_{y}n),   \\
C&=(c_{x}+\gamma_{x} m,c_{y}+\gamma_{y}n)
\end{aligned}
\]
for some  $\alpha_x, \alpha_y, \beta_x, \beta_y, \gamma_x, \gamma_y \in \Z$.

We get
\[
\begin{aligned}
{0} &= {D(A,B,C)}
={ D(a,b,c)}
+{nD((a_{x},\alpha_{2}),(b_{x},\beta_{y}),(c_{x},\gamma_{y}))}\\
&+{ mD((\alpha_{x},a_{y}),(\beta_{x},b_{y}),(\gamma_{x},c_{y} ))}
+{mnD((\alpha_{x},\alpha_{y}),(\beta_{x},\beta_{y}),(\gamma_{x},\gamma_{y}))}.\\
\end{aligned}
\]
Hence $D(a,b,c)\equiv 0 \pmod{\gcd(m,n)}.$
\end{proof}

\begin{proof}[Proof of Theorem \ref{M2}(2)]
Let $\gcd(m,n)= 2.$ Let 
\[
X=\left\{(0,0),(0,1),(1,0),(1,1)\right\}\subset T_{m\times n}.
\]
It is easy to check that $D(a,b,c)\equiv \pm 1 \pmod {\gcd(m,n)}$ for any $a,b,c\in X.$ 
By Lemma~\ref{L:D}, $X$ satisfies the no-three-in-line condition. 
Thus $\uptau\left(T_{m\times n}\right)\geq 4.$
Now Theorem \ref{M1} finishes  the statement.
\end{proof}

\begin{proof}[Proof of Theorem \ref{M2}(3a)]
Let $p=\gcd(m,n)$ be an odd prime.  
Assume without loss of generality that $\gcd(pm,n)=p^2$. Consequently $m = p k$ and $ n = p^2 l$ for some positive integers $k,l$.
Define $X=\{ (i, i^2 p) \in T_{m\times n}: i\in \pset \}$ and
$Y=\{ (i, i^2 p + 1) \in T_{m\times n}: i\in \pset \}$. 
We will show that the set $X\cup Y$ of $2p$ points satisfies the no-three-in-line condition.

Take any three distinct points $(i,i^2 p), (j,j^2 p), (k,k^2 p)$ from $X$.
We will show that these points are not in a line on $T_{m\times n}$.
To do this, we will show that three points 
$A=(i + \alpha_{x} m, i^2 p + \alpha_{y} n),$
$B=(j + \beta_{x} m,  j^2 p + \beta_{y} n),$ 
$C=(k + \gamma_{x} m, k^2 p + \gamma_{y} n),$  
where $\alpha_{x},$ $\alpha_{y},$ $\beta_{x},$ $\beta_{y},$ $\gamma_{x},$ $\gamma_{y}\in \Z$ 
are not in a line on $\ZZ$.
We get
\[
\begin{aligned}
{D(A,B,C)}
&=
D((i,i^2 p),(j,j^2 p),(k,k^2 p))
+{nD((i,\alpha_{y}),(j,\beta_{y}),(k,\gamma_{y}))}\\
&+{mD((\alpha_{x},i^2 p),(\beta_{x},j^2 p),(\gamma_{x},k^2 p))}
+{mnD((\alpha_{x},\alpha_{y}),(\beta_{x},\beta_{y}),(\gamma_{x},\gamma_{y}))}\\
&
=p\cdot D((i,i^2),(j,j^2),(k,k^2))
+p^2 l \cdot D((i,\alpha_{y}),(j,\beta_{y}),(k,\gamma_{y}))\\
&
+pk \cdot p \cdot D((\alpha_{x},i^2),(\beta_{x},j^2),(\gamma_{x},k^2))
+pk \cdot p^2 l \cdot D((\alpha_{x},\alpha_{y}),(\beta_{x},\beta_{y}),(\gamma_{x},\gamma_{y}))\\
&=p (j-i)(k-i)(k-j) + p^2 M \neq 0,\\
\end{aligned}
\]
since $p\nmid (j-i)(k-i)(k-j)$ and $M\in \Z$.

In the same way it can be shown that
any three points from $Y$ are not in a line on $T_{m\times n}$.

Now take any two points $(i,i^2 p), (j,j^2 p)$ from $X$ and $(k,k^2 p+1)$ from $Y$. We have
\[
D((i,i^2 p), (j,j^2 p), (k,k^2 p+1)) \equiv j - i \pmod{ \gcd(m,n) }.
\]
By Lemma~\ref{L:D} these points are not in a line.
The same argument works if we take one point from $X$ and any two points from $Y$. We showed that $\uptau\left(T_{m\times n}\right)\geq 2 p.$ 
Now Theorem \ref{M1} gives the statement.
\end{proof}

\section{Proof of Theorem \ref{M2}(3b)}

Let $p=\gcd(m,n)$ and $\rho:T_{m\times n}\to T_{p\times p}$ be the projection defined as follows
$\rho( u, v ) = (u \bmod p, v \bmod p)$. Since $p|m$ and $p|n$, the diagram
\[
\xymatrix{
\ZZ \ar[d]_{\pi_{p,p}} \ar[r]^{\pi_{m,n}}
&T_{m\times n} \ar[ld]^\rho\\
T_{p\times p}&}\eqno (1)
\]
commutes. Consequently, 
the image of every line on $T_{m\times n}$ is a line on $T_{p\times p}$.
This immediately implies the following result.

\begin{lemma} \label{L:epi}
Let $p=\gcd(m,n)$. The following holds:
\begin{enumerate}
	\item $\uptau\left(T_{m\times n}\right) \geq \uptau\left(T_{p\times p}\right)$,
	\item  If the preimage of every line on $T_{p\times p}$ is a line on $T_{m\times n}$,
	then $\uptau\left(T_{m\times n}\right)=\uptau\left(T_{p\times p}\right)$.  
\end{enumerate}
\end{lemma}

Now, let us define the family 
$\mathcal P_O = \left\{\ell^{\beta}:\beta \in \pset \right\}\cup \left\{\ell^\infty\right\}$ of lines on $T_{p\times p}$ 
passing through  $O=(0,0)$, where
\[
\begin{aligned} 
\ell^{\beta} &= \left\{\pi_{p,p}(k,\beta k) \in T_{p\times p} :k\in \Z  \right\}, \\
\ell^\infty &= \left\{\pi_{p,p}(0, k)\in T_{p\times p}         :k\in\Z \right\}.
\end{aligned}
\]

The following  results can also be found in \cite{Fowler}.

\begin{lemma} \label{L:sum}
Let p be an odd prime. Then $T_{p\times p} = \bigcup_{\ell\in \mathcal P_{O}} \ell$.

\end{lemma}
\begin{proof}
Consider $(a_{x},a_{y})\in T_{p\times p}.$ Suppose $a_{x}\neq 0$. Since $p$ is an odd prime, 
there is a unique $\beta \in \pset$ such that  $(a_{x},a_{y})\in \ell^{\beta}$. 
If $a_{x}$ is zero, then $(a_{x},a_{y})\in \ell^{\infty}.$ Hence, $T_{p\times p}\subset \bigcup_{\ell\in \mathcal P_{O}} \ell.$
The inclusion 
$ \bigcup_{\ell\in \mathcal P_{O}} \ell \subset T_{p\times p}$ is obvious.

\end{proof}

In the next two lemmas we will investigate the sets $\rho^{-1}(\ell)$ for $\ell\in\mathcal P_{O}.$

\begin{lemma} \label{L:pencil}
Let $p$ be an odd prime. 
For every $\beta \in \{1,2,\ldots, p-1 \}$ the set $\rho^{-1}(\ell^\beta)$ is a line on 
$T_{m\times n}$. 
\end{lemma}
\begin{proof}
First we claim that there is $\alpha \in \Z$ such that $\alpha \equiv \beta \pmod p$ 
and $\gcd(\alpha m,n)=p$. 
Indeed, for instance, take $\alpha$ such that the following conditions are satisfied:
(i)~$\alpha=\beta+kp$ for some $k\in \Z$, 
(ii)~$\alpha$ is a prime, 
(iii)~$\alpha$ is greater than $n$.
Since $p$ is an odd prime, the existence of such $\alpha$ is guaranteed 
by   Dirichlet's theorem on arithmetic progressions.

Define $L^\alpha = \{ \pi_{m,n}(k,\alpha k):k\in\Z\}$.
Now, we will show that $\rho^{-1}(\ell^\beta)=L^{\alpha}.$ Let $(a_{x},a_{y})\in \rho^{-1}(\ell^\beta)$.  
Then $a_{y}\equiv \beta a_{x} \pmod {p}$ and  $a_{y}\equiv  \alpha a_{x} \pmod {p}$.
By Theorem \ref{T:CRT} there exists $k_{1}\in \Z$ such that
\[
\begin{aligned}
k_{1}&\equiv \alpha a_{x}   &\pmod {\alpha m}, \\
k_{1}&\equiv  a_{y}    &\pmod {n}.
\end{aligned}
\]
It is easy to see that $k_{1}=\alpha k$ for $k\in \Z$ and we get
\[
\begin{aligned}
\alpha  k &\equiv \alpha a_{x}   &\pmod {\alpha m}, \\
\alpha  k &\equiv a_{y}     &\pmod {n}.
\end{aligned}
\]
Hence
\[
\begin{aligned}
k &\equiv  a_{x}   &\pmod {m}, \\
\alpha  k &\equiv a_{y}  &\pmod {n}.
\end{aligned}
\]
This means that 
$
(a_x, a_y)\in L^{\alpha}.
$
Since $\rho(L^\alpha)\subset \ell^\alpha=l^\beta$, we have $L^{\alpha}\subset \rho^{-1}(\ell^\beta)$.  
The proof is finished.

\end{proof}

\begin{lemma} \label{L:hv}
Let $p=\gcd(m,n)$. The following holds:
\begin{enumerate}
\item  If  $\gcd(m,pn)=p$,
then $\rho^{-1}(\ell^\infty)$ is a line in $T_{m\times n}$,
\item  If  $\gcd(pm,n)=p$,
then $\rho^{-1}(\ell^0)$ is a line in $T_{m\times n}$.
\end{enumerate}
\end{lemma}

\begin{proof}(1) 
Let $L^\infty = \{\pi_{m,n}(pk,k):k\in\Z\}$.
We will show that $\rho^{-1}(\ell^\infty) = L^\infty$.
Take $(a_x, a_y)\in \rho^{-1}(\ell^\infty)$.
Hence $a_x \equiv 0 \pmod p$.
By Theorem \ref{T:CRT}   there exists  $k_{1}\in \Z$ such that
\[
\begin{aligned}
k_{1}&\equiv a_{x}    &\pmod {m}, \\
k_{1}&\equiv pa_{y}   &\pmod {pn}.
\end{aligned}
\]
It is easy to see that $k_{1}=pk$ for some  $k\in \Z$ and we get
\[
\begin{aligned}
pk &\equiv a_{x}    &\pmod {m}, \\
pk &\equiv pa_{y}  &\pmod {pn}.
\end{aligned}
\]
Hence
\[
\begin{aligned}
pk&\equiv a_{x}    &\pmod m, \\
k&\equiv a_{y}   &\pmod n
\end{aligned}
\]
and $a\in L^\infty.$ The inclusion 
$L^\infty\subset \rho^{-1}(\ell^\infty)$ is obvious.
The proof is finished.

\noindent (2) The proof is similar to (1). 
\end{proof}

\begin{theorem} \label{L:rev}
Let $p=\gcd(m,n)$ be an odd prime such that $\gcd(pm,n)=\gcd(m,pn)=p$. Then
we have $\uptau\left(T_{m\times n}\right) = \uptau\left(T_{p\times p}\right).$
\end{theorem}
\begin{proof}
By Lemma \ref{L:hv} and \ref{L:pencil} we get that $\rho^{-1}(\ell)$ is a line on $T_{m\times n}$ for any 
$\ell\in\mathcal P_{O}$.
Hence the preimage of every line on $T_{p\times p}$ is a line on $T_{m\times n}.$ 
 Lemma \ref{L:epi}(2) finishes the proof.
\end{proof}

The following result can be found in \cite{Fowler}. Here we present the complete proof.

\begin{theorem} \label{T:p_p}
Let $p$ be an odd prime. Then $\uptau(T_{p\times p}) = p+1$.
\end{theorem}
\begin{proof}
Let $p$ be an odd prime number. 
If $p\equiv 1 \pmod 4$ then take $q$ to be some quadratic nonresidue modulo $p$.
If $p\equiv 3 \pmod 4$ take $q$ to be some quadratic residue modulo $p$.
Define $X=\{(x,y)\in T_{p\times p} : x^2 + q \cdot y^2 \equiv 1 \pmod p \}$.
It is known that $X$ has $p+1$ points. See for example Theorem~10.5.1 in~\cite{BEW}.
By Lagrange theorem for congruences, any line intersects $X$ in at most two points.
Hence the set $X$ satisfies the no-three-in-line condition and $\uptau(T_{p\times p}) \geq p+1.$

Let $X$ satisfy the no-three-in-line condition. We can assume that $O\in X$. By Lemma \ref{L:sum}, 
every other point of the $T_{p\times p}$ lies on one or the
other of the $p+1$ lines passing through $O$. Hence $|X|\leq p+2$ and we have $\uptau(T_{p\times p})\leq p+2$.

Assume that there exists a set $Y$ with $p+2$ points which satisfies the no-three-in-line condition.
Take any line $L$ in $T_{p\times p}$. We claim that either $|L\cap Y|=0$ or $|L\cap Y|=2$.
Indeed, if $L\cap Y = \{y\}$ then $L$ is the line passing through point $y\in Y$ which does not pass by any other point of $Y$ and consequently $|Y|\leq p+1$, a contradiction.
Now fix any point $z$ not in $Y$. Since each line through~$z$ contains either $0$ or $2$ points of $Y$, the number of points in $Y$ is even, a contradiction with the fact that $p+2$ is odd.
This means that the set $Y$ does not exist.
\end{proof}

\begin{proof}[Proof of Theorem \ref{M2}(3b)]
\item Theorem \ref{L:rev} together with Theorem \ref{T:p_p} gives the statement.
\end{proof}

\subsection*{Acknowledgment}
We express our sincere thanks to the anonymous
referee for his/her valuable advice which resulted in an improvement 
of this article.


\begin{thebibliography}{00}

\bibitem{BEW}
B. C. Berndt, R. J. Evans, K. S. Williams,
Gauss and Jacobi sums,
Wiley, 1998.


\bibitem{Dud}
H. E. Dudeney, 
Amusements in Mathematics, 
Nelson, Edinburgh 1917, pp. 94, 222.

 
\bibitem{E} 
P. Erd\"os. Appendix, in K.F. Roth, 
On a problem of Heilbronn,
J. London Math. Soc. 26, 198--204, 1951.

\bibitem{F92}
A. Flammenkamp,
Progress in the no-three-in-line problem,
J. Combin. Theory Ser. A, 60(2), 305--311, 1992.

\bibitem{F98}
A. Flammenkamp,
Progress in the no-three-in-line problem {II},
J. Combin. Theory Ser. A, 81(1), 108--113, 1998. 

\bibitem{Fowler} 
J. Fowler, A. Groot, D. Pandya, B. Snapp,
The no-three-in-line problem on a torus,
arXiv:1203.6604v1.

\bibitem{GK}
R. K. Guy P. A. Kelly, 
The No-Three-Line Problem, 
Math. Bull. Vol. 11, pp. 527-531, 1968.

\bibitem{HJSW}
R. R. Hall, T. H. Jackson, A. Sudbery, and K. Wild,
Some advances in the no-three-in-line problem,
J. Combinatorial Theory Ser. A, 18:336--341, 1975.


\bibitem{PA}
A. Por, D.R. Wood,
No-Three-in-Line-in-3D,
Algorithmica 47(4), 481--488, 2007.


\bibitem{SSZ} Z. St\c{e}pie\'{n}, L. Szymaszkiewicz, M. Zwierzchowski,
The Cartesian product of cycles with small 2-rainbow domination number, 
Journal of Combinatorial Optimization, doi:10.1007/s10878-013-9658-0.
\end{thebibliography}
\end{document}